\documentclass[11pt,reqno]{amsart}

   \input{preamble}
\renewcommand{\epsilon}{\varepsilon}
\renewcommand{\hat}{\widehat}
\renewcommand{\tilde}{\widetilde}
\renewcommand{\bar}{\overline}

\DeclareMathOperator{\supp}{supp}

\def\pt{\partial_t}
\def\Laq{(-\Delta)^{\frac{1}{4}}}
\def\Lah{(-\Delta)^{\frac{1}{2}}}

\newcommand{\besov}{\dot{B}^{\frac{n}{2}}_{2,1}}

\newcommand{\R}{\mathbb{R}}

\newcommand{\Hy}{\mathbb{H}^2}
\newcommand{\be}{\begin{equation}}
\newcommand{\ee}{\end{equation}}

\def\XXint#1#2#3{{\setbox0=\hbox{$#1{#2#3}{\int}$}
\vcenter{\hbox{$#2#3$}}\kern-.5\wd0}}

         \title[]{Global Weak Solutions for the Half-Wave Maps Equation in $\R$}
         \date{}
\author[]{Yang Liu}
\address{B\^atiment des Math\'ematiques \\ EPFL \\ Station 8 \\ 1015 Lausanne \\ Switzerland}
\email{y.liu@epfl.ch}
        
\begin{document}
        \maketitle
\begin{abstract}
    We establish the existence of weak global solutions of the half-wave maps equation with the target $S^2$ on $\R^{1+1}$ with large initial data in $\dot{H}^1 \cap \dot{H}^{\frac{1}{2}}(\R)$. We first prove the global well-posedness of a regularized equation. Then we show that the weak limit of the regularized solutions is a weak solution of the half-wave maps equation as the regularization parameter $\varepsilon \rightarrow 0$.
\end{abstract}        
\section{Introduction}

Let $u: \R^{1+1}\rightarrow S^2 \subseteq \R^3$ be smooth and bounded with the property that $\nabla_{t,x}u(t,\cdot)\in L^r(\R^n)$ for some $r\in (1,\infty)$, and furthermore $\lim_{|x|\rightarrow +\infty}u(t, x) = Q$ for some fixed $Q\in M$, for each $t$. We define the operator $\Lah u = -\sum_{j=1}^n (-\triangle)^{-\frac12}\partial_j (\partial_j u)$. The Cauchy problem of {\it{half-wave map}} is given by

\begin{equation}
    \label{eq3:halfwave}
\begin{cases}
    &\partial_t u  = u\times \Lah u:=f(x,t) \\
    &u(0,x)= u_0: \R \rightarrow S^2,
\end{cases}
\end{equation}
where $u_0 \in \dot{H}^1 \cap \dot{H}^{\frac{1}{2}}(\R, S^2)$ is a smooth and constant outside of a compact domain (this ensures that $\Lah u_0$ is well-defined).


We consider the Cauchy problem \eqref{eq3:halfwave} with large data in $\dot{H}^1 \cap \dot{H}^{\frac{1}{2}}(\R)$. We show that there exists a weak solution to the half-wave equation \eqref{eq3:halfwave} in $L_{t,loc}^2([0,\infty),$ $ \dot{H}^{\frac{1}{2}}(\R, S^2))$ for smooth initial data $u_0 \in \dot{H}^1 \cap \dot{H}^{\frac{1}{2}} (\R, S^2)$.

We say $u$ is a global weak solution for \eqref{eq3:halfwave} if it satisfies

\begin{equation}
    \label{weaksol}
        -\int_0^\infty \int_{\R} u \cdot \varphi_t dx dt - \int_{\R} u_0 \varphi(x) dx = \int_0^\infty \int_{\R} \Laq (u  \times \varphi ) \cdot \Laq u \  dx dt,
\end{equation}

for all $\varphi \in C_c^\infty(\R, S^2).$

\begin{theorem}
    \label{thm:main}
   Let $u_0 \in \dot{H}^1 \cap \dot{H}^{\frac{1}{2}} (\R, S^2)$ be a smooth initial data and constant outside of a compact domain s.t. $\lim_{|x|\rightarrow \infty} \ u_0(x) =Q$ for a fixed $Q\in S^2$. Then the Cauchy problem \eqref{eq3:halfwave} admits a global weak solution  $u \in L_{t,loc}^2([0,\infty),\dot{H}^{\frac{1}{2}}(\R, S^2))$ such that $\lim_{|x|\rightarrow \infty} \ u(t,x) =Q$ for all $t\in [0,\infty)$.
   
\end{theorem}

We introduce the following parabolic regularization of \eqref{eq3:halfwave}: 
\begin{equation}
    \label{eq3:reg}
    u_t-\varepsilon \triangle u = u \times \Lah u.   
\end{equation}

Using the theory of parabolic equations, we can establish the existence of classical global solutions  $u_\varepsilon \in L_t^\infty L_x^\infty \cap C_t^0([0,\infty), \dot{H}^{\frac{1}{2}}(\R))$ for \eqref{eq3:reg}. Then we can show that the weak limit of these regularized solutions is a weak solution of the half-wave maps equation as $\varepsilon \rightarrow 0$ as defined in \eqref{weaksol}. 

In this work, we only consider initial data, which is smooth and constant outside of a compact domain, since we need the extra regularity to establish the classical well-posedness theory of \eqref{eq3:reg} and the converge of the solution $u$ to the fixed point $Q$. We expect the theorem holds for more general initial data in $ \dot{H}^{\frac{1}{2}}(\R)$. We leave this for future work.

\section{Background}
The half-wave equation is related to the well-studied Schrödinger maps equation in the form of
\[
u_t = u \times \triangle u,    
\]
and the classical wave equation
\[
\Box u=\partial_\alpha \partial^\alpha u=-u \partial_\alpha u^T \partial^\alpha u.
\]

Moreover, we can also view the half-wave map equation as the Landau-Lifshitz equation

\[
u_t=u \times \Lah u +\lambda u \times (u \times \Lah u)
\]

without the Gilbert damping term as $\lambda \rightarrow 0$.

The weak solution of the half-wave map equation with torus domain $\mathbb{T}^n$ was studied in the works of \cite{kato2006quasi},\cite{pu2011fractional} and \cite{pu2013well} in the context of the well-posedness problem of the fractional Landau-Lifshitz equation without Gilbert damping. Pu and Guo in \cite{pu2013well} established the weak solutions of the half-wave map equation with torus domain $\mathbb{T}^n$ via the vanishing viscosity method and Kato's method.

The half-wave map equation \eqref{eq3:halfwave} admits a conserved energy

\begin{equation}
\label{intro:eq:energy}
E(t):=\int_{\R^n} |\Laq u|^2 dx,
\end{equation}
where $\Laq u := -\sum_{j=1}^n (-\triangle)^{-\frac34}\partial_j (\partial_j u)$. 

This gives the a priori condition that $u(t)\in \dot{H}^{\frac{1}{2}}(\R^n)$ which implies that the half-wave maps is energy-critical when $n=1$. In the work of \cite{LenzmannEnno2018Oehm}, Lenzmann and Schikorra give a full classification of the traveling solitary waves for the energy-critical problem with target $S^2$ for $n=1$.

It is also worth noting that the critical points of the \eqref{intro:eq:energy} are the $\frac{1}{2}-$harmonic maps. The fractional harmonic maps were studied in the works of \cite{da2011sub}, \cite{francesca2011three} and \cite{DaLioFrancesca2016HaM}.

The one-dimensional energy-critical half-wave maps are of notable physics interest, intensively studied in the works of \cite{berntson2020multi, enno2018energy, gerard2018lax, lenzmann2018short, zhou2015solitons}. The one-dimensional half-wave maps arise as a continuum limit of the discrete {\em Calogero-Moser (CM) spin system.} Interested readers shall refer to \cite{lenzmann2020derivation} for the derivation of the half-wave maps equation from the CM system. For more on the CM systems, we refer the reader to \cite{blom1999finding, gibbons1984generalisation} in which the authors study the theory of completely integrable systems. In addition, the classical CM spin systems can be obtained by taking a suitable semiclassical limit of the quantum spin chains related to the well-known {\em Haldane-Shastry (HS) spin chains}, see e.g.\cite{haldane1988exact, shastry1988exact}, which are exactly solvable quantum models. \par

The global well-posedness of the Cauchy problem \eqref{eq3:halfwave} with target $S^2$ for small $\besov\times \dot{B}^{\frac{n}{2}-1}_{2,1}$ initial data was established by Krieger and Sire \cite{krieger2017small} for $n\geq 5$. The result was later improved by Kiesenhofer and Krieger \cite{2019arXiv190412709K} to $n=4$. In previous work \cite{liu2023global}, the author established global well-posedness for the half-wave map with $S^2$ target for small $\dot{H}^{\frac{n}{2}} \times \dot{H}^{\frac{n}{2}-1}$ initial data for $n\geq 5$. Global well-posedness for the equation with $\Hy$ target for small smooth $\besov \times \dot{B}^{\frac{n}{2}-1}_{2,1}$ initial data was also proven in \cite{liu2023global}. These works are based on the strategy that transforms the \eqref{eq3:halfwave} to a nonlinear wave equation form. One can then utilize well-established theory for wave maps ( e.g. \cite{sterbenz2005global, T1}) to show global well-posedness for the half-wave maps equation with small initial data.

\section{Regularization of Half-Wave Map Equation}

For $\varepsilon,T>0$, we define the regularized half-wave map equation as
\begin{equation}
    \label{eq:reg}
    \begin{cases}
        &\partial_t u_{\varepsilon} - \varepsilon \triangle u_{\varepsilon} = u_{\varepsilon}\times \Lah u_{\varepsilon}:=N(u_\varepsilon)(x,t) \\
        &u_{\varepsilon}(0,x)= u_0: \R \rightarrow S^2
    \end{cases}
\end{equation}


After the regularization, the equation  \eqref{eq:reg} is a nonlinear parabolic equation. We define its corresponding fundamental solution $K_\varepsilon(x,t)$ for each $\varepsilon$ as
\begin{equation}
    \label{eq:K}
    K_{\varepsilon}(x,t)=\begin{cases}
        \int_\R e^{2\pi i x \cdot \xi -\varepsilon |\xi|^2 t} d\xi=\frac{1}{\sqrt{4 \varepsilon \pi t}} e^{-\frac{|x|^2}{4\varepsilon t}}  \quad &\text{if } t>0 \\
        0 &\text{if } t=0
    \end{cases}
\end{equation}

Moreover, we know that $\hat{K}_{\varepsilon}(\xi,t)= e^{-\varepsilon |\xi|^2 t}$ and $\int_{\R} K_{\varepsilon}(x,t) \, dx =1$.

In particular, $K_{\varepsilon}(x,t)$ is a solution of 
\begin{equation}
    \begin{cases}
        &\partial_t K_{\varepsilon}(x,t)  -\varepsilon \triangle K_{\varepsilon}(x,t)=0 \\
        &K_{\varepsilon}(x,0)=\delta(x)
    \end{cases}
\end{equation}

By the Duhamel principle, we can define a solution $u_{\varepsilon}$ of \eqref{eq:reg} as
\begin{equation}
    \label{solv}
        u_\varepsilon (t,x) =\int_0^t \int_\R K_{\varepsilon}(x-y,t-s) \ N(u_\varepsilon)(y,s) \, dy \, ds +\int_\R K_{\varepsilon}(x-y,t) u_0(y) \, dy  
\end{equation}



We first use an iteration scheme to find a local solution $u_\varepsilon$ of \eqref{eq:reg} in $  C_0^t \dot{H}^1 \cap L_t^\infty L_x^\infty ([0, T]\times \R,\R^3 )$.

\begin{theorem}{}
    \label{thm:local_existence}
    Let $T, \varepsilon>0$, and $u_0 \in \dot{H}^1 \cap \dot{H}^{\frac{1}{2}}(\R)$ smooth and constant outside of a compact set of  $\R$.
    Then there exists a maximal time $T(u_0)>0$ such that \eqref{eq:reg} admit an unique solution  $u_\varepsilon$ in $C_0^t \dot{H}_x^1 \cap L_t^\infty L_x^\infty ([0,T]\times \R,\R^3 )$ such that $\lim_{|x|\rightarrow \infty} \ u_\varepsilon(t,x) =Q$ for all $t\in [0,T)$.
\end{theorem}
\begin{proof}
    We use an iteration scheme to show there exists a solution for equation \eqref{eq:reg} in the solution space $X_T=L_t^\infty L_x^\infty \cap C_t^0 \dot{H}_x^1([0,T]\times \R)$.

We start with $u_{\varepsilon}^{(0)}$ which solves the homogeneous equation:
\begin{equation}
    \label{eq:iter0}
    \begin{cases}
        &\partial_t u_{\varepsilon}^{(0)} -\varepsilon \triangle u_{\varepsilon}^{(0)} =0\\
        &u_{\varepsilon}^{(0)}(0,\cdot)=u_0
    \end{cases}
\end{equation}

By the fundamental solution, we know that $u_\varepsilon^{(0)}(t,x)= K_\varepsilon \star u_0(t,x)= \int_{\R} K_\varepsilon(t,x-y) u_0(y) \, dy$. Hence, we have
\begin{equation*}
    \begin{split}
         \|u_\varepsilon^{(0)}\|_{L_t^\infty L_x^\infty}
        \leq \|u_0\|_{L_t^\infty L_x^\infty} | \int_{\R} K_\varepsilon(t,x-y) \, dy| \leq  \|u_0\|_{L_t^\infty L_x^\infty} 
    \end{split}
\end{equation*}

Hence we know that 
\begin{equation}
    \|u_\varepsilon^{(0)}\|_{L_t^\infty L_x^\infty} \leq \|u_0\|_{L_t^\infty L_x^\infty} 
\end{equation}

Next,we consider the $C_0^t \dot{H}^1$ norm of $u_\varepsilon^{(0)}$.
\begin{equation}
    \begin{split}
        \|u_\varepsilon^{(0)}(t,\cdot)\|_{\dot{H}^1} &= \| \int_{\R} K_\varepsilon(t,y) \nabla_x u_0(x-y) \, dy \|_{L_x^2} \\
        &\leq  \|K_\varepsilon(t-s,\cdot)\|_{L_x^1} \| \nabla_x u_0(s,\cdot) \|_{L_x^2} \, \\
        &\leq  \| u_0\|_{L_t^\infty \dot{H}^1}  
    \end{split}
\end{equation}

Hence, we have
\begin{equation}
    \|u_\varepsilon^{(0)}\|_{L_t^\infty \dot{H}^1} \leq \| u_0\|_{L_t^\infty \dot{H}^1} 
\end{equation}

Therefore, we conclude that $u_{\varepsilon}^{(0)} \in X_T$. Moreover, $u_\varepsilon^{(0)}$ is the unique solution for \eqref{eq:iter0} by the uniqueness of the homogeneous heat equation theory.
Next, we define $u_\varepsilon^{(1)}$ as a solution of 
\begin{equation}
    \label{eq:iter1}
    \begin{cases}
        &\partial_t u_\varepsilon^{(1)} -\varepsilon \triangle u_\varepsilon^{(1)} =u_{\varepsilon}^{(0)} \times \Lah u_{\varepsilon}^{(0)}:=N(u_{\varepsilon}^{(0)})\\
        &u_\varepsilon^{(1)}(0,\cdot)=u_0
    \end{cases}
\end{equation}

By the Duhamel's principle, we have
\[
    u_\varepsilon^{(1)}= K_\varepsilon \star N(u_\varepsilon^{(0)})+ K_\varepsilon \star u_0.
\]  

First of all, we know that 
\begin{equation*}
    \begin{split}
        &\|K_\varepsilon \star N(u_\varepsilon^{(0)})(t,\cdot)\|_{L_x^\infty}\\
         &\leq \int_0^t \frac{1}{\sqrt{4\varepsilon\pi (t-s)}} \ \int_{\R} e^{-\frac{|x-y|^2}{4\varepsilon(t-s)}} |u_\varepsilon^{(0)}\times \Lah u_\varepsilon^{(0)}|(s,y) \, dy \, ds \\ 
        &\leq \|u_\varepsilon^{(0)}\|_{L_t^\infty L_x^\infty} \int_0^t \frac{1}{\sqrt{4\pi \varepsilon (t-s)}} \ \| e^{-\frac{|x-\cdot|^2}{4\varepsilon(t-s)}} \|_{L^2_y} \ \|\Lah u_\varepsilon^{(0)}(s,\cdot)\|_{L_y^2} \, ds \\
        &\lesssim \|u_\varepsilon^{(0)}\|_{L_t^\infty \dot{H}^1} \int_0^t \frac{1}{(t-s)^{\frac{1}{4}}} \, ds\\
        &\lesssim t^{\frac{3}{4}} \ \|u_\varepsilon^{(0)}\|_{L_t^\infty \dot{H}^1} 
    \end{split}
\end{equation*}

Therefore, we have
\begin{equation}
    \max_{t\in [0,T]} \|u_\varepsilon^{(1)}\|_{L_x^\infty} \leq C T^{\frac{3}{4}} \|u_\varepsilon^{(0)}\|_{L_t^\infty \dot{H}^1} 
\end{equation}

For the $C_0^t \dot{H}^1$ norm, we use Minkowski's inequality for integrals to derive:
\begin{equation*}
    \begin{split}
      &\| K_\varepsilon \star N(u_\varepsilon^{(0)})(t,\cdot)\|_{\dot{H}_x^1} \\
       &=\| \nabla_x K_\varepsilon \star N(u_\varepsilon^{(0)})(t,\cdot)\|_{L_x^2}\\
       &\leq  \int_0^t \|\nabla_x K_\varepsilon \star N(u_\varepsilon^{(0)})(t,s) \|_{L_x^2}\, ds.
    \end{split}
\end{equation*}

Through Young's inequality, we have
\begin{equation*}
    \begin{split}
      &\|\nabla_x K_\varepsilon \star N(u_\varepsilon^{(0)})(t,s) \|_{L_x^2}\\
      &\leq \|\nabla_x K_\varepsilon(t-s)\|_{L_x^1} \|N(u_\varepsilon^{(0)})(s)\|_{L_x^2}
    \end{split}
\end{equation*}

We know that
\begin{equation*}
    \begin{split}
        &\|N(u_\varepsilon^{(0)})(s)\|_{L_x^2} \\
        =& \int_{\R} |u_\varepsilon^{(0)}\times \Lah u_\varepsilon^{(0)}|^2 \, dx\\
        \leq& \int_{\R} |u_\varepsilon^{(0)}|^2 |\Lah u_\varepsilon^{(0)}|^2 \, dx + \int_{\R} |u_\varepsilon^{(0)} \cdot \Lah u_\varepsilon^{(0)}|^2 \, dx\\
        \lesssim& \|u_\varepsilon^{(0)}\|_{L_t^\infty \dot{H}_x^1}^2,
    \end{split}
\end{equation*}

and 
\begin{equation*}
    \begin{split}
        &\|\nabla_x K_\varepsilon(t-s)\|_{L^1_x}\\
        &=\frac{1}{\sqrt{\pi \varepsilon (t-s)}} \frac{1}{2\varepsilon(t-s)} \ \int_{\R} |x-y| e^{-\frac{|x-y|^2}{4\varepsilon(t-s)}}  \, dx \\
        &\leq \frac{2}{\sqrt{\pi (t-s)}}.
    \end{split}
\end{equation*}

Thus we have 
\begin{equation}
    \label{eq:iter1_1}
    \begin{split}
        \| K_\varepsilon \star N(u_\varepsilon^{(0)})\|_{L_t^\infty \dot{H}_x^1} &\leq C \sup_{t\in [0,T]} \int_0^t \frac{1}{\sqrt{\pi (t-s)}} \, ds \leq C \sqrt{T}
    \end{split}
\end{equation}

Therefore, we know that solution $u_\varepsilon^{(1)} \in X_T$ as well.

We then show that solution $u_\varepsilon^{(1)}$ is the unique solution for \eqref{eq:iter1} in $X_T$ space. Assume we have another solution $v_\varepsilon^{(1)} \in X_T$ for \eqref{eq:iter1}. We consider $w_\varepsilon=u_\varepsilon^{(1)}-v_\varepsilon^{(1)}$ and we have
\begin{equation}
    \begin{cases}
        &\partial_t w_\varepsilon -\varepsilon \triangle w_\varepsilon =0\\
        &w_\varepsilon(\cdot,0)=0
    \end{cases}
\end{equation}

For the energy functional $E(t)=\int_{\R} |w_\varepsilon|^2 dx$. We have

\begin{equation}
    \frac{d}{dt} \int_{\R} |w_\varepsilon|^2 dx =-\varepsilon \int_{\R} |\nabla_x w_\varepsilon|^2 dx \leq 0 
\end{equation}

Hence 
\begin{equation}
    \frac{d}{dt} E(t)\leq 0
\end{equation}
    
 Therefore, we know that $E(t)\equiv 0$ for all $t\in [0,T]$. Thus $w_\varepsilon(t,\cdot)=0$ for all $t\in [0,T]$. So we have $u_\varepsilon^{(1)}=v_\varepsilon^{(1)}$ on $[0,T]\times \R$. Therefore, $u_\varepsilon^{(1)}$ is the unique solution for \eqref{eq:iter1} in $X_T$.

For $j \geq 2$, we define the general iterative scheme inductively as follows: 
\begin{equation*}
    \begin{cases}
        &\partial_t u_\varepsilon^{(j)} -\varepsilon \triangle u_\varepsilon^{(j)} =u_\varepsilon^{(j-1)} \times \Lah u_\varepsilon^{(j-1)}:=N(u_\varepsilon^{(j-1)})\\
        &u_\varepsilon^{(j)}(0,\cdot)=u_0
    \end{cases}
\end{equation*}

We have  
\[
    u_\varepsilon^{(j)}= K_\varepsilon \star N(u_\varepsilon^{(j-1)})+ K \star u_0 .
\]

Following the same procedure, we have a unique solution $u_\varepsilon^{(j)} \in X_T$. Therefore, we conclude that 
\begin{equation}
    u_\varepsilon^{(j)} \in X_T \quad \forall j \geq 1.
\end{equation}

Next, we want to show that $\{u_\varepsilon^{(j)}\}$ is a Cauchy sequence in $X_T$. 

For any $k,l$ large enough, we have
\begin{equation}
    \begin{split}
         \|u_\varepsilon^{j}-u_\varepsilon^{k}\|_{L_t^\infty \dot{H}^1}=& \|K_\varepsilon \star ( N(u_\varepsilon^{(j-1)}) -N(u_\varepsilon^{(k-1)}))\|_{L_t^\infty \dot{H}_x^1}\\
        \leq& \|K \star ( (u_\varepsilon^{(j-1)} -u_\varepsilon^{(k-1)})\times u_\varepsilon^{(j-1)})\|_{L_t^\infty \dot{H}_x^1}\\
        &+\|K_\varepsilon \star ( u_\varepsilon^{(k-1)} \times \Lah(u_\varepsilon^{(j-1)} - u_\varepsilon^{(k-1)}))\|_{L_t^\infty \dot{H}_x^1}
    \end{split}
\end{equation}

We estimate the two terms in the above equation separately. For the first term, we have
\begin{equation}
    \begin{split}
        &\|K_\varepsilon \star ( (u_\varepsilon^{(j-1)} -u_\varepsilon^{(k-1)})\times \Lah u_\varepsilon^{(j-1)})\|_{L_t^\infty \dot{H}^1} \\
        &\leq \int_0^T \| \nabla_x K_\varepsilon(t-s)\|_{L_x^1} \|(u_\varepsilon^{(j-1)} -u_\varepsilon^{(k-1)})\times \Lah u_\varepsilon^{(j-1)}(\cdot,s)\|_{L_x^2} ds\\
        &\leq C \sqrt{T} \| u_\varepsilon^{(j-1)}\|_{L_t^\infty \dot{H}^1} \| u_\varepsilon^{(j-1)}- u_\varepsilon^{(k-1)}\|_{L_t^\infty L_x^\infty}
    \end{split}
\end{equation}

Similarly, for the second term, we have 
\begin{equation}
    \begin{split}
        &\|K_\varepsilon \star ( u_\varepsilon^{(k-1)} \times \Lah(u_\varepsilon^{(j-1)} - u_\varepsilon^{(k-1)}))\|_{L_t^\infty \dot{H}_x^1} \\
        &\leq \int_0^T \| \nabla_x K_\varepsilon(t-s)\|_{L_x^1} \|u_\varepsilon^{(k-1)} \times \Lah(u_\varepsilon^{(j-1)} - u_\varepsilon^{(k-1)})\|_{L_x^2} \\
        &\leq C \sqrt{T} \|u_\varepsilon^{(k-1)}\|_{L_t^\infty L_x^\infty} \| u_\varepsilon^{(j-1)}- u_\varepsilon^{(k-1)}\|_{L_t^\infty \dot{H}^1} 
    \end{split}
\end{equation}

Hence, we know that 
\begin{equation}
    \|u_\varepsilon^{(j)}-u_\varepsilon^{(k)}\|_{L_t^\infty \dot{H}^1} \leq C \sqrt{T} \| u_\varepsilon^{(j-1)}- u_\varepsilon^{(k-1)}\|_{L_t^\infty \dot{H}^1},
\end{equation}

where $C$ is a constant independent of $\|u_\varepsilon^{(k-1)}\|_{L_t^\infty L_x^\infty}, \| u_\varepsilon^{(j-1)}\|_{L_t^\infty \dot{H}^1}$. Moreover,
\begin{equation}
    \begin{split}
        &\|u_\varepsilon^{(j)}-u_\varepsilon^{(k)}(t,\cdot)\|_{L_x^\infty} \\
        \leq& \int_0^t \frac{1}{\sqrt{4\pi (t-s)}} \ \int_{\R} e^{-\frac{|x-y|^2}{4\varepsilon(t-s)}} |N(u_\varepsilon^{j})-N(u_\varepsilon^{(k)}|(y,s) \, dy \, ds \\ 
        \leq& \int_0^t \frac{1}{\sqrt{4\pi (t-s)}} \ \int_{\R} e^{-\frac{|x-y|^2}{4\varepsilon(t-s)}} |(u_\varepsilon^{(j-1)} -u_\varepsilon^{(k-1)})\times \Lah u_\varepsilon^{(j-1)}|(y,s) \, dy \, ds \\
        &+\int_0^t \frac{1}{\sqrt{4\pi (t-s)}} \ \int_{\R} e^{-\frac{|x-y|^2}{4\varepsilon(t-s)}} |u_\varepsilon^{(k-1)} \times \Lah(u_\varepsilon^{(j-1)} - u_\varepsilon^{(k-1)})|(y,s) \, dy \, ds \\
        \leq& \|u_\varepsilon^{(j-1)}-u_\varepsilon^{(k-1)}\|_{L_t^\infty L_x^\infty} \int_0^t \frac{1}{\sqrt{4\pi (t-s)}} \ \| e^{-\frac{|x-\cdot|^2}{4\varepsilon(t-s)}} \|_{L^2_y} \ \|\Lah v^{(j-1)}(s,\cdot)\|_{L_y^2} \, ds \\
        &+ \|u_\varepsilon^{(k-1)}\|_{L_t^\infty L_x^\infty} \int_0^t \frac{1}{\sqrt{4\pi (t-s)}} \ \| e^{-\frac{|x-\cdot|^2}{4\varepsilon(t-s)}} \|_{L^2_y} \ \|\Lah (u_\varepsilon^{(j-1)}-u_\varepsilon^{(k-1)})(s,\cdot)\|_{L_y^2} \, ds \\
        \leq& C \|u_\varepsilon^{(j-1)}-u_\varepsilon^{(k-1)}\|_{L_t^\infty \dot{H}^1} \int_0^t \frac{1}{(t-s)^{\frac{1}{4}}} \, ds\\
        \leq& C t^{\frac{3}{4}} \ \|u_\varepsilon^{(j-1)}-u_\varepsilon^{(k-1)}\|_{L_t^\infty L_x^\infty}
        \end{split}
\end{equation}

Hence, 
\begin{equation}
        \|u_\varepsilon^{(j)}-u_\varepsilon^{(k)}\|_{L_t^\infty L_x^\infty} \leq C T^{\frac{3}{4}} \| u_\varepsilon^{(j-1)}- u_\varepsilon^{(k-1)}\|_{L_t^\infty L_x^\infty}
\end{equation}

Therefore, we have 
\begin{equation}
    \|u_\varepsilon^{(j)}-u_\varepsilon^{(k)}\|_{L_t^\infty L_x^\infty \cap C_0^t \dot{H}^1} \leq C T^{\alpha} \| u_\varepsilon^{(j-1)}- u_\varepsilon^{(k-1)}\|_{L_t^\infty L_x^\infty \cap C_0^t \dot{H}^1},
\end{equation}
for some $\alpha \in (0,1)$ and $C$ depends on $\|u_\varepsilon^{(j-1)}\|_{X_T}, \|u_\varepsilon^{(k-1)}\|_{X_T}.$

We choose $T>0$ small enough so that $C T^{\alpha} < 1$. Then we have a contraction mapping from $L_t^\infty L_x^\infty \cap C_0^t \dot{H}^1$ to itself. Thus, we conclude that there exists a $u_{\varepsilon} \in L_t^\infty L_x^\infty \cap C_0^t \dot{H}^1$ s.t. $u_{\varepsilon}^{(j)} \rightarrow u_{\varepsilon}$ in $L_t^\infty L_x^\infty \cap C_0^t \dot{H}^1$ as $j \rightarrow \infty$. 

Moreover, $u_\varepsilon$ is the solution for \eqref{eq:reg} in $X_T$ and it has the form:
\[
 u_\varepsilon = K_\varepsilon \star N(u_\varepsilon)+ K_\varepsilon \star u_0    
\]

Furthermore, we have 
\begin{equation}
    \label{eq:localu}
    \|u_\varepsilon\|_{L_t^\infty L_x^\infty \cap C_0^t \dot{H}^1} \leq C T^{\alpha} \|u_0\|_{L_t^\infty L_x^\infty \cap C_0^t \dot{H}^1}
\end{equation}

Hence, we showed there exists a unique solution $u_\varepsilon$ for \eqref{eq:reg} in $L_t^\infty L_x^\infty \cap C_0^t \dot{H}^1([0,T]\times \R)$.
\end{proof}
\begin{Lemma}
    For the solution $u_\varepsilon$ we obtained, we have 
    \begin{equation}
        \lim_{|x|\rightarrow \infty} u_\varepsilon(t,x)=Q, \ \forall t\in [0,T).
    \end{equation}
\end{Lemma}
\begin{proof}
We show the convergence iteratively for each $u^{(j)}_\varepsilon$. Hence the limit $u_\varepsilon$ also converges to $Q$ as $|x|\rightarrow \infty$.

For $t\in [0,T)$ and $x$ large, we first have 
\begin{align}
|x \cdot (u_\varepsilon^{(0)} -Q)| 
&\leq \Big|\int_{\R} \frac{1}{\sqrt{4 \varepsilon \pi t}} e^{-\frac{|x-y|^2}{4\varepsilon t}} x (u_0(y)-Q) \, dy  \Big| \nonumber \\
&\leq \Big|\int_{\R} \frac{1}{\sqrt{4 \varepsilon \pi t}} e^{-\frac{|x-y|^2}{4\varepsilon t}} (x-y)(u_0(y)-Q) \, dy \Big| \label{covu01} \\
&+\Big|\int_{\R} \frac{1}{\sqrt{4 \varepsilon \pi t}} e^{-\frac{|x-y|^2}{4\varepsilon t}} y (u_0(y)-Q) \, dy \Big| \label{covu02}
\end{align}

For \eqref{covu01}, we have
\begin{equation*}
    \begin{split}
        &\int_{\R} \frac{1}{\sqrt{4 \varepsilon \pi t}} e^{-\frac{|x-y|^2}{4\varepsilon t}} |x-y| |u_0(y)-Q| \, dy \\
    &\leq C  \|e^{-\frac{|x-y|^2}{4\varepsilon t}} |x-y|\|_{L_y^1} \|u_0-Q\|_{L_y^\infty} \\
    &\leq C
    \end{split}
\end{equation*}

For \eqref{covu02}, we have 
\begin{equation*}
    \begin{split}
        &\int_{\R} \frac{1}{\sqrt{4 \varepsilon \pi t}} e^{-\frac{|x-y|^2}{4\varepsilon t}} |y| |u_0(y)-Q| \, dy \\
        &\leq  C \| y (u_0(y)-Q)\|_{L_y^\infty}.
    \end{split} 
\end{equation*}

Since $u_0$ is constant outside of a compact domain, we know that $y (u_0(y)-Q)$ is bounded. Hence we have
\begin{equation}
    \|x \cdot (u_\varepsilon^{(0)} -Q)\|_{L_x^\infty}  \leq C.
\end{equation}

So we know that 
\[
    \lim_{|x|\rightarrow \infty} u_\varepsilon^{(0)}=Q, \ \text{for all} \ t\in [0,T).
\]

Next, we consider $u_\varepsilon^{(1)}$. We have
\begin{equation*}
    \begin{split}
        \lim_{|x|\rightarrow \infty} |u_\varepsilon^{(1)}(t,x)-Q| &\leq  \lim_{|x|\rightarrow \infty} |K_\varepsilon \star N(u_\varepsilon^{(0)})| + \lim_{|x|\rightarrow \infty}|K_\varepsilon \star u_0 -Q|.
    \end{split}
\end{equation*}

We already know that $\lim_{|x|\rightarrow \infty}|K_\varepsilon \star u_0 -Q|=0$. We then show that
\[
    \lim_{|x|\rightarrow \infty} K_\varepsilon \star N(u_\varepsilon^{(0)})=0.
\]

We first show that
\begin{equation}
    \label{cov:nu0}
    |x\cdot K_\varepsilon \star N(u_\varepsilon^{(0)})| \leq C, \ \text{for large} \ x.
\end{equation} 

We have 
\begin{align}
        |x\cdot K_\varepsilon \star N(u_\varepsilon^{(0)})|&= \Big|\int_0^t \frac{1}{\sqrt{4\varepsilon\pi (t-s)}} \ \int_{\R} e^{-\frac{|x-y|^2}{4\varepsilon(t-s)}} x (u_\varepsilon^{(0)}\times \Lah u_\varepsilon^{(0)}) (s,y) \, dy \, ds \Big| \nonumber \\
        &\leq \Big| \int_0^t \frac{1}{\sqrt{4\varepsilon\pi (t-s)}} \ \int_{\R} e^{-\frac{|x-y|^2}{4\varepsilon(t-s)}} (x-y)  (u_\varepsilon^{(0)}\times \Lah u_\varepsilon^{(0)}) (s,y)  \, dy \, ds \Big| \label{covu11}\\
        &+ \Big|\int_0^t \frac{1}{\sqrt{4\varepsilon\pi (t-s)}} \ \int_{\R} e^{-\frac{|x-y|^2}{4\varepsilon(t-s)}} y  (u_\varepsilon^{(0)}\times \Lah u_\varepsilon^{(0)}) (s,y) \, dy \, ds\Big| \label{covu12}
\end{align}

For \eqref{covu11}, we have 
\begin{equation*}
    \begin{split}
        &\Big| \int_0^t \frac{1}{\sqrt{4\varepsilon\pi (t-s)}} \ \int_{\R} e^{-\frac{|x-y|^2}{4\varepsilon(t-s)}} (x-y)  (u_\varepsilon^{(0)}\times \Lah u_\varepsilon^{(0)}) (s,y)  \, dy \, ds\Big| \\
    &\leq \int_0^t \frac{1}{\sqrt{4\varepsilon\pi (t-s)}} \|e^{-\frac{|x-y|^2}{4\varepsilon(t-s)}} (x-y)\|_{L_y^2} \|u_\varepsilon^{(0)}\|_{L_y^\infty}  \| \Lah u_\varepsilon^{(0)}\|_{L_y^2} \, ds \\
    &\leq C \sqrt{t} \|u_\varepsilon^{(0)}\|_{\dot{H^1}} \\
    &\leq C.
    \end{split} 
\end{equation*}

Then we consider \eqref{covu12}. We have
\begin{equation*}
    \begin{split}
        &\Big|\int_0^t \frac{1}{\sqrt{4\varepsilon\pi (t-s)}} \ \int_{\R} e^{-\frac{|x-y|^2}{4\varepsilon(t-s)}} y  (u_\varepsilon^{(0)}\times \Lah u_\varepsilon^{(0)}) (s,y) \, dy \, ds\Big| \\
        &\leq \int_0^t \frac{1}{\sqrt{4\varepsilon\pi (t-s)}} \|e^{-\frac{|x-y|^2}{4\varepsilon(t-s)}}\|_{L_y^1} \|u_\varepsilon^{(0)}\|_{L_y^\infty}  \|y \Lah u_\varepsilon^{(0)}\|_{L_y^\infty} \, ds \\
        &\leq C \sqrt{t} \|y \Lah u_\varepsilon^{(0)}\|_{L_y^\infty} 
    \end{split}
\end{equation*}

So we reduce the problem to show that
\begin{equation}
    \label{covyu0}
    \|y \Lah u_\varepsilon^{(0)}(y)\|_{L_y^\infty}  \leq C.
\end{equation}
We use integration by parts to derive
    \begin{equation*}
        \begin{split}
            |x \partial_x u_\varepsilon^{(0)}(t,x)| &= |\frac{1}{\sqrt{4 \varepsilon \pi t}} \int_{\R} x \partial_x e^{-\frac{|x-y|^2}{4\varepsilon t}} u_\varepsilon^{(0)}(y) \, dy |\\
            &= |\frac{1}{\sqrt{4 \varepsilon \pi t}} \int_{\R} x \partial_y e^{-\frac{|x-y|^2}{4\varepsilon t}}  u_0(y) \, dy |\\
            &=|\frac{1}{\sqrt{4 \varepsilon \pi t}} \int_{\R} e^{-\frac{|x-y|^2}{4\varepsilon t}} x \partial_y u_0(y) \, dy| \\
            &\leq |\frac{1}{\sqrt{4 \varepsilon \pi t}} \int_{\R} e^{-\frac{|x-y|^2}{4\varepsilon t}} (x-y) \partial_y  u_0(y) \, dy |\\
            &+|\frac{1}{\sqrt{4 \varepsilon \pi t}} \int_{\R} e^{-\frac{|x-y|^2}{4\varepsilon t}} y \partial_y  u_0(y) \, dy|
        \end{split}  
    \end{equation*}

So we have 
\begin{equation}
    \begin{split}
        &|\frac{1}{\sqrt{4 \varepsilon \pi t}} \int_{\R} e^{-\frac{|x-y|^2}{4\varepsilon t}} (x-y) \partial_y  u_0(y) \, dy | \\
        &\leq C \|u_0\|_{\dot{H}^1} \|e^{-\frac{|x-y|^2}{4\varepsilon t}} (x-y)\|_{L_y^2} \\
        &\leq C.
    \end{split} 
\end{equation}

For the second term, we have
\begin{equation}
    \begin{split}
        &|\frac{1}{\sqrt{4 \varepsilon \pi t}} \int_{\R} e^{-\frac{|x-y|^2}{4\varepsilon t}} y \partial_y  u_0(y) \, dy | \\
        &\leq C \|y \partial_y u_0\|_{L_y^\infty} \|e^{-\frac{|x-y|^2}{4\varepsilon t}}\|_{L_y^1} \\
        &\leq C \|y \partial_y u_0\|_{L_y^\infty} \\
        &\leq C, 
    \end{split}
\end{equation}
since $\partial_y u_0$ is compact supported.
Hence we have 
\begin{equation*}
    |x\partial_x u_\varepsilon^{(0)}(t,x)| \leq C,
\end{equation*}

We can iteratively consider the quantity $x^n \partial_x u_\varepsilon^{(0)}(t,x)$ for $n\geq 1$ and we have
\begin{equation}
    \label{cov:u0}
    |x^n \partial_x u_\varepsilon^{(0)}(t,x)| \leq C.
\end{equation} 

So we know that $\partial_x u_\varepsilon^{(0)}(t,x)$ is decay arbitrarily fast as $|x|\rightarrow \infty$.

We write 
\begin{equation}
    \Lah u_\varepsilon^{(0)}=(-\triangle)^{-\frac{1}{2}} \partial_x \partial_x u_\varepsilon^{(0)}.
\end{equation}

Note here that $(-\triangle)^{-\frac{1}{2}} \partial_x$ is a Hilbert transform via a Fourier multiplier $i\text{sgn}(\xi)$, so we have
\begin{equation}
    \Lah u_\varepsilon^{(0)}=\frac{1}{\pi} \text{p.v.} \ \int_{-\infty}^{\infty} \frac{\partial_x u_\varepsilon^{(0)}(t,y)}{x-y}\,dy.
\end{equation}

Therefore, for $x\cdot \Lah u_\varepsilon^{(0)}(t,x)$,  we have
\begin{equation}
    \label{cov:u01}
    \begin{split}
        |x\cdot \Lah u_\varepsilon^{(0)}(t,x)| &=\Big|\frac{1}{\pi} \  \text{p.v.} \ \int_{-\infty}^{\infty} \frac{x \partial_y u_\varepsilon^{(0)}(t,y)}{x-y}\,dy \Big| \\
        &\leq C \Big|\text{p.v.} \ \int_{-\infty}^{\infty} \frac{1}{(x-y)x^n}\,dy \Big|\\
        &\leq C.
    \end{split}
\end{equation}
Moreover, we can also conclude that $|x^n \cdot \Lah u_\varepsilon^{(0)}(t,x)|$ is bounded for large $x$. Hence $\Lah u_\varepsilon^{(0)}(t,x)$ is decay arbitrarily fast as $|x|\rightarrow \infty$. 

Thus we have
\begin{equation*}
    \lim_{|x|\rightarrow 0} u_\varepsilon^{(1)} =Q.
\end{equation*}

    From above argument, we know that $\lim_{|x|\rightarrow \infty} u_\varepsilon^{(j)} =Q$ holds whence 
    \begin{equation}
        \label{covyuj}
        \| x \cdot \Lah u_\varepsilon^{(j-1)}(x)\|_{L_x^\infty} \leq C.
    \end{equation}

    We show \eqref{covyuj} by iteration again. We first show that 
    \begin{equation}
        \label{covyu1}
        \|x \cdot \Lah u_\varepsilon^{(1)}\|_{L_x^\infty} \leq C.
    \end{equation}

    We know that 
    \begin{equation*}
        \Lah u_\varepsilon^{(1)} \simeq  K_\varepsilon \star \Lah (N(u_\varepsilon^{(0)})) +\Lah u_\varepsilon^{(0)}.
    \end{equation*}

    With \eqref{cov:u01}, we only need to show that
    \begin{equation*}
        \|x \cdot K_\varepsilon \star \Lah (N(u_\varepsilon^{(0)}))\|_{L_x^\infty} \leq C.
    \end{equation*}

We have
\begin{equation*}
    \begin{split}
       \Lah N(u_\varepsilon^{(0)}) (t,x)
       &=\frac{1}{\pi} \  \text{p.v.} \ \int_{-\infty}^{\infty} \frac{ \partial_y N(u_\varepsilon^{(0)})(t,y)}{x-y}\,dy \\
    \end{split}
\end{equation*}

We know that
\begin{equation*}
    \begin{split}
        \partial_y N(u_\varepsilon^{(0)}) &=\partial_y u_\varepsilon^{(0)} \times \Lah u_\varepsilon^{(0)}+ u_\varepsilon^{(0)} \times \partial_y \Lah u_\varepsilon^{(0)} \\
    \end{split} 
\end{equation*}

As showed in \eqref{cov:u0} and \eqref{cov:u01}, we know that $\partial_y u_\varepsilon^{(0)}$ and $\Lah u_\varepsilon^{(0)}$ decay arbitrarily fast as $|x|\rightarrow \infty$. Hence $ \partial_y N(u_\varepsilon^{(0)})$ is also decay arbitrarily fast as $|x|\rightarrow \infty$. So we can conclude that 
\begin{equation}
   |x^n \cdot \Lah N(u_\varepsilon^{(0)}) |\leq C, \ \text{for } n \geq 1.
\end{equation}

Thus we have that 
\begin{align}
        &|x \cdot  (K_\varepsilon \star \Lah N(u_\varepsilon^{(0)}))| \nonumber \\
        &\leq | \int_0^t \frac{1}{\sqrt{4\varepsilon\pi (t-s)}} \int_{\R} e^{-\frac{|x-y|^2}{4\varepsilon(t-s)}} x  \Lah N(u_\varepsilon^{(0)})(s,y)\, dy ds | \nonumber \\
        &\leq | \int_0^t \frac{1}{\sqrt{4\varepsilon\pi (t-s)}} \int_{\R} e^{-\frac{|x-y|^2}{4\varepsilon(t-s)}} (x-y)  \Lah N(u_\varepsilon^{(0)})(s,y)\, dy ds | \label{covu21}\\
        &+| \int_0^t \frac{1}{\sqrt{4\varepsilon\pi (t-s)}} \int_{\R} e^{-\frac{|x-y|^2}{4\varepsilon(t-s)}} y  \Lah N(u_\varepsilon^{(0)})(s,y)\, dy ds | \label{covu22}
\end{align}

For \eqref{covu21}, we have 
\begin{equation*}
    \begin{split}
        &| \int_0^t \frac{1}{\sqrt{4\varepsilon\pi (t-s)}} \int_{\R} e^{-\frac{|x-y|^2}{4\varepsilon(t-s)}} (x-y)  \Lah N(u_\varepsilon^{(0)})(s,y)\, dy ds | \\
        &\leq \int_0^t \frac{1}{\sqrt{4\varepsilon\pi (t-s)}} \|e^{-\frac{|x-y|^2}{4\varepsilon(t-s)}} (x-y)\|_{L_y^2} \|u_\varepsilon^{(0)}\|_{L_y^\infty}  \|\Lah u_\varepsilon^{(0)}\|_{L_y^2}\, ds \\
        &\leq C \sqrt{t} \|u_\varepsilon^{(0)}\|_{\dot{H^1}} \\
        &\leq C.
    \end{split}
\end{equation*}

For \eqref{covu22}, we have 
\begin{equation*}
    \begin{split}
        &| \int_0^t \frac{1}{\sqrt{4\varepsilon\pi (t-s)}} \int_{\R} e^{-\frac{|x-y|^2}{4\varepsilon(t-s)}} y  \Lah N(u_\varepsilon^{(0)})(s,y)\, dy ds | \\
        &\leq \int_0^t \frac{1}{\sqrt{4\varepsilon\pi (t-s)}} \|e^{-\frac{|x-y|^2}{4\varepsilon(t-s)}}\|_{L_y^1} \|u_\varepsilon^{(0)}\|_{L_y^\infty}  \|y \Lah N(u_\varepsilon^{(0)})\|_{L_y^\infty}\, ds \\
        &\leq C \sqrt{t} \|y \Lah N(u_\varepsilon^{(0)})\|_{L_y^\infty} \\
        &\leq C.
    \end{split}
\end{equation*}

Hence we have
\begin{equation*}
    |x \cdot \Lah(K_\varepsilon \star N(u_\varepsilon^{(0)})(t,x)| \leq C.
\end{equation*}

So we showed that \eqref{covyu1} holds. Then we use the same argument to show that 
\begin{equation*}
    \| x \cdot \Lah u_\varepsilon^{(j)}\|_{L_x^\infty} \leq C,
\end{equation*}
holds for all $j\geq 1$ to conclude the Lemma.
\end{proof}

\section{Global Well-Posedness for the Regularized Equation}
We consider the solution $u_\varepsilon$ for the regularized equation \eqref{eq:reg}. Although the solution $u_{\varepsilon}$ no longer map into $S^2$, it will still be bounded in $L_t^\infty L_x^\infty [0,T]\times \R$. We consider $v_{\varepsilon}=u_\varepsilon \cdot u_\varepsilon$.  We have $v_\varepsilon$ satisfies the following equation:
\begin{equation}
    \begin{split}
        \partial_t (u_{\varepsilon}\cdot u_{\varepsilon})&= 2u_{\varepsilon} \cdot (\varepsilon \triangle u_{\varepsilon} + u_{\varepsilon}\times \Lah u_{\varepsilon}) \\
       &=2 \varepsilon \ u_{\varepsilon}\cdot \triangle u_{\varepsilon} \\
       &= \varepsilon \triangle (u_{\varepsilon}\cdot u_{\varepsilon})  -2\varepsilon |\nabla u_{\varepsilon}|^2
    \end{split}
\end{equation}

Hence 
\begin{equation}
    \label{eq:v}
    \begin{cases}
        &\partial_t v_{\varepsilon} -\varepsilon \triangle v_{\varepsilon} =- 2\varepsilon |\nabla u_{\varepsilon}|^2 \leq 0 \\
        &v_{\varepsilon}(0,\cdot)= u_0 \cdot u_0=1
    \end{cases}
\end{equation}


We want to show that we have a priori bound for $v_{\varepsilon}$ in $L_t^\infty L_x^\infty ([0,T]\times \R)$.

\begin{Proposition}{Local Maximal Princinple}
    Let $\Omega\subseteq \R$ be an open and bounded set, and $u\in C^2(\Omega_T) \cap C^0 (\bar{\Omega}_T)$, where $\Omega_T=[0,T]\times \Omega$ satisfies:
    \begin{equation}
        \pt u -\triangle u \leq 0 \quad \text{in} \quad \Omega_T
    \end{equation}
    
    Then we have 
    \begin{equation}
        \sup_{\bar{\Omega}_T} u =\sup_{\partial \Omega_T} u,
    \end{equation}
    where $\partial \Omega_T:=([0,T]\times \partial \Omega) \cup (\{0\}\times \Omega)$.
\end{Proposition}

\begin{proof}
    It suffices to show the result for all $T' < T$. For each $T'$, we assume the maximum is attained at $(t_0,x_0)$. 

    We first consider the case when $\triangle u > \pt u$ in $\Omega_T$. We show that an interior maximum cannot be attained.

    Assume $(t_0,x_0) \notin \partial \Omega_T$, then at $(t_0,x_0)$, we have 
    \begin{equation}
        \begin{cases}
            \pt u(t_0) \geq 0 \\
            \triangle u(t_0,x_0) \leq 0
        \end{cases}
    \end{equation}

    Hence, we have $\triangle u(t_0,x_0) \leq \pt u(t_0)$. This contradicts the assumption that $\triangle u > \pt u$ in $\Omega_T$. Hence, we have $(t_0,x_0) \in \partial \Omega_T$.

    For the $\triangle u=\pt u$, we consider $u_{\lambda}(t,x)= u(t,x)-\lambda t$, for $\lambda >0$. We see that 
    \begin{equation*}
        \pt u_{\lambda} -\triangle u_{\lambda} =\pt u -\triangle u -\lambda < 0 \quad \text{in} \quad \Omega_T
    \end{equation*}

    Hence by the strict case above, we obtain
    \[
        \sup_{\bar{\Omega}_T} u_{\lambda}= \sup_{\partial \Omega_T} u_{\lambda}.
    \]
    
    Then we can conclude that
    \[
        \sup_{\bar{\Omega}_T} u \leq \sup_{\bar{\Omega}_T} u_{\lambda}+\lambda T =\sup_{\partial \Omega_T} u_{\lambda} +\lambda T \leq \sup_{\partial \Omega_T} u  +\lambda T
    \]
    
    Let $\lambda \rightarrow 0$, we have 
    \[
        \sup_{\bar{\Omega}_T} u = \sup_{\partial \Omega_T} u.
    \]
\end{proof}

\begin{theorem}{Global Maximal Principle}
    \label{thm:maximum_v}
    For $T>0$, and a smooth function $u_\varepsilon:[0,T] \times \mathbb{R} \rightarrow \mathbb{R}^3$ solves \eqref{eq:reg}. We further assume that there's a point $Q \in S^2$, s.t. $u_\varepsilon(t,x) \rightarrow Q$ as $|x| \rightarrow \infty$ for every $t>0$.
    
    Then $v_\varepsilon:=u_\varepsilon \cdot u_{\varepsilon}$ satisfies \eqref{eq:v} and we have 
    \begin{equation}
        \max_{t\in [0,T]} \|v_\varepsilon\|_{L^\infty} \leq C
    \end{equation}
        
\end{theorem}
\begin{proof}
    For any $R>0$, we consider the time slab $\Omega_{T,R}=(0,T) \times B_R$, where $B_R=(-R,R)$. On $\Omega_{T,R}$, we know that $v_\varepsilon$ satisfies
    \begin{equation*}
        \pt v_\varepsilon - \varepsilon \triangle v_\varepsilon \leq 0.
    \end{equation*}
    
    Hence by the local maximum principle, we have 
    \begin{equation}
        |v_\varepsilon| \leq \max_{\partial \Omega_{T,R}} |v_\varepsilon| = 1
    \end{equation}


    Furthermore, We know that $u_\varepsilon(t,x) \rightarrow Q$ as $|x|\rightarrow \infty$, hence $v_\varepsilon(t,x) \rightarrow 1$ as $|x|\rightarrow \infty$ for every $t>0$. Therefore, for any $\lambda >0$, there existed a $R>0$, s.t. 
    \begin{equation}
        |v_\varepsilon| \leq 1+\lambda \quad \text{on} \ [0,T]\times (\R \setminus B_R).
    \end{equation}

    Let $\lambda \rightarrow 0$, we conclude that
    \begin{equation}
        \max_{t\in [0,T]} \|v_\varepsilon\|_{L_x^\infty(\R)} \leq 1
    \end{equation}
    
\end{proof}

\begin{theorem}
    \label{thm:globalh1}
    For $u_\varepsilon$ defined by Theorem \ref{thm:local_existence}, we can extend the solution to a global solution  $u_{\varepsilon} \in L_t^\infty L_x^\infty \cap C_t^0([0,\infty), \dot{H}^1(\R)) $.
\end{theorem}
\begin{proof}
We consider the energy term
\begin{equation}
    E(u_\varepsilon)=\frac{1}{2} \int_{\mathbb{R}} |\nabla u_\varepsilon|^2 \, dx.
\end{equation}

We have 
\begin{equation}
    \begin{split}
        &\frac{d}{dt} \Big(\frac{1}{2} \int_{\mathbb{R}} |\nabla u_\varepsilon|^2 \, dx \Big) \\
        &=\int \nabla \pt u_\varepsilon \cdot \nabla u_\varepsilon \, dx \\
        &=-\varepsilon \int | \triangle u_\varepsilon|^2 dx - \int(u_\varepsilon \times \Lah u_\varepsilon) \cdot \triangle u_\varepsilon \, dx 
    \end{split}
\end{equation}

As $u_\varepsilon$ is bounded by Theorem \ref{thm:maximum_v}, we use the Cauchy-Schwarz inequality to conclude that 
\begin{equation}
    \begin{split}
        &|\int(u_\varepsilon \times \Lah u_\varepsilon) \cdot \triangle u_\varepsilon \, dx| 
        \\ &\leq \varepsilon \int | \triangle u_\varepsilon|^2 dx +\frac{C}{\varepsilon} \int | \nabla u_\varepsilon|^2 dx, \\
    \end{split}
\end{equation}
where $C$ is a constant depending on $\|u_\varepsilon(0,\cdot)\|_{L^\infty}$.

Therefore, we have
\begin{equation*}
    \partial_t E(u_\varepsilon) \leq \frac{C}{\varepsilon} E(u_\varepsilon).
\end{equation*}

By Gronwall inequality, we conclude that
\begin{equation}
    \label{eq:energy}
    E(u_\varepsilon(t)) \leq \frac{C}{\varepsilon} e^{\frac{C}{\varepsilon} t} \quad \forall t\geq 0.
\end{equation}

Therefore, we know that 
\begin{equation}
    \|u\|_{X_T} \leq \frac{C}{\varepsilon} e^{\frac{C}{\varepsilon} t}
\end{equation}

Let's denote $T_{\max}$ as the maximum existence time of the solution $u_{\varepsilon}$. We then show that $T_{\max} =\infty$ by contradiction.

Assume that $T_{\max} < \infty$, for solution $u_{\varepsilon}$ on $[0,T_{\max}]\times \R$, we then consider the Cauchy problem with initial data $\tilde{u}_{\varepsilon}(0,\cdot)=u_{\varepsilon}(T_{\max}-\delta)$ for a small $\delta>0$. 

The local existence of solution from Theorem \ref{thm:local_existence} implies that there exists a unique solution $\tilde{u}_{\varepsilon}(t)$ in $L_t^\infty L_x^\infty \cap C_t^0([T_{\max}-\delta, \tilde{T} ]\times \R)$ for some $\tilde{T}>T_{\max}$ and small $\delta>0$. The uniqueness of the solution implies that $u_\varepsilon$ and $\tilde{u}$ are the same on $[T_{\max}-\delta, T_{\max}]\times \R$. Thus the maximum existence time $T_{\max}=\infty$. So the solution $u_{\varepsilon}$ is a global solution.

\end{proof}

Moreover, the global $\dot{H}^{1}$ can be further extend to a global $\dot{H}^{\frac{1}{2}}$ solution.

\begin{theorem}
    \label{thm:globalh12}
    For $T, \varepsilon>0$, and a smooth initial data $u_0 \in \dot{H}^1 \cap \dot{H}^{\frac{1}{2}} (\R, S^2)$ which is constant outside of a compact set. The Cauchy problem of \eqref{eq:reg} admits a unique solution  $u_{\varepsilon} \in L_t^\infty L_x^\infty \cap C_t^0([0,\infty), \dot{H}^{\frac{1}{2}}(\R)) $ such that $\lim_{|x|\rightarrow \infty} \ u(t,x) =Q$ for a fixed $Q\in S^2$.
\end{theorem}

\begin{proof}
    

We now consider the case where the initial data $u_0$ is in $\dot{H}^{\frac{1}{2}}$. 
For the solution $u_\varepsilon= K_\varepsilon \star N(u_\varepsilon)+ K_\varepsilon \star u_0$, we first show that $K_\varepsilon \star u_0$ is in $C_t^0([0,\infty), \dot{H}^{\frac{1}{2}}(\R))$.  We have

\begin{equation}
    \begin{split}
        &\|K_\varepsilon \star u_0(t,\cdot)\|_{L_t^\infty \dot{H}_x^{\frac{1}{2}}} \\
        &= \||\xi|^{\frac{1}{2}} \hat{K}_\varepsilon(t,\xi) \hat{u}_0(\xi) \|_{L_t^\infty L_{\xi}^2} \\
        &=\| |\xi|^{\frac{1}{2}} e^{-\varepsilon|\xi|^2 t} \hat{u}_0(\xi) \|_{L_t^\infty L_{\xi}^2} \\
        &\leq \|u_0\|_{\dot{H}^{\frac{1}{2}}}
    \end{split}
\end{equation}

Next, we consider the $K_\varepsilon \star N(u_\varepsilon)$ part.  We first use the interpolation theorem to show that it is in $ C_{t,loc}^0([0,\infty), \dot{H}^{\frac{1}{2}}(\R))$.

We know that 
\[
\| f\|_{\dot{H}^{\frac{1}{2}} }  \leq \|f\|^{\frac{1}{2}}_{L^2} \| f\|^{\frac{1}{2}}_{\dot{H}^1},
 \]

 thus we only need to show that $ K_\varepsilon \star N(u_\varepsilon) \in L_x^2$. We have

\begin{equation*}
    \begin{split}
      &\| K_\varepsilon \star N(u_\varepsilon)(t,\cdot)\|_{L_x^2} \\
       &=\| \int_0^t \int_{\R} K_\varepsilon(t-s,x-y) (u \times \Lah u)(s,y) \, dy ds\|_{L_x^2}\\
       &\leq \int_0^t \|K_\varepsilon(t-s)\|_{L_x^1} \|N(u_\varepsilon)(s)\|_{L_x^2} \ ds\\
       &\leq t \|u_\varepsilon\|_{L_t^\infty \dot{H}_x^1}
    \end{split}
\end{equation*}

Hence we know that $K_\varepsilon \star N(u_\varepsilon)(t,\cdot) \in L_x^2(\R)$. So we conclude that $K_\varepsilon \star N(u_\varepsilon)(t,\cdot) \in \dot{H}_x^{\frac{1}{2}}(\R)$. Thus we know that $u_\varepsilon$ is in $C_{t,loc}^0([0,\infty), \dot{H}_x^{\frac{1}{2}}(\R))$.


For a solution $u_\varepsilon(t) \in \dot{H}_x^{\frac{1}{2}}$ of \eqref{eq:reg}, we consider the critical energy 

\begin{equation}
    \label{eq:critical}
    \begin{split}
        E_c (u_\varepsilon(t))&= \frac{1}{2} \int_{\mathbb{R}} |\Laq u_\varepsilon(t)|^2 \, dx.
    \end{split}
\end{equation}

We have 
\begin{equation}
    \begin{split}
        &\frac{d}{dt} \frac{1}{2} \int_{\mathbb{R}} |\Laq u_\varepsilon|^2 \\
        &=\int \pt u_\varepsilon \cdot \Lah u_\varepsilon \, dx \\
        &= \int \varepsilon \triangle u \cdot \Lah u_\varepsilon \, dx \\
        &=- \varepsilon \int |(-\triangle)^\frac{3}{4} u_\varepsilon|^2 \, dx \\
        &\leq 0
    \end{split}
\end{equation}

Hence we know that 
\begin{equation}
    \label{eq:h12uni}
    \sup_{t\in [0,\infty)}\|u_\varepsilon\|_{\dot{H}^{\frac{1}{2}}} \leq \|u_0\|_{\dot{H}^{\frac{1}{2}}}.
\end{equation}

Hence the local solution $u_\varepsilon$ is also a global solution of \eqref{eq:reg} in $C_{t}^0([0,\infty), \dot{H}_x^{\frac{1}{2}}(\R))$.
\end{proof}

Furthermore, we also know that $u_{\varepsilon}(t,\cdot) \in L_{x,loc}^2(\R)$ by the following proposition.
\begin{Proposition}[Proposition 1.37 in \cite{bahouri2011fourier}]
    \label{prop:h12}
    For $s\in (0,1)$, and $u \in \dot{H}^s(\R^n)$. Then $u \in L_{loc}^2(\R^n)$.
\end{Proposition}

\section{Weak Solution}
    We know that for a global weak solution of \eqref{eq:reg} $u_{\varepsilon}$, it satisfies

\begin{equation}
    \label{eq:weaksol}
    \begin{split}
        &-\int_0^\infty \int_{\R} u_{\varepsilon} \cdot \varphi_t dx dt - \int_{\R} u_0 \varphi(x) dx \\
        &= \varepsilon \int_0^\infty \int_{\R} u_{\varepsilon}  \cdot \triangle \varphi dx dt + \int_0^\infty \int_{\R} \Laq (u_{\varepsilon}  \times \varphi) \cdot \Laq u_{\varepsilon} dx dt,
    \end{split}
\end{equation}
for any $\varphi \in C_c^\infty \cap \mathcal{S}(\R, \R^3)$.

We want to show that as $\varepsilon \rightarrow 0$ (up to a subsequence), we have a weak solution $u_{\star} \in L_{t,loc}^2([0,\infty),\dot{H}^{\frac{1}{2}}(\R))$ for \eqref{eq3:halfwave} s.t.

\begin{equation}
    \begin{split}
        -\int_0^\infty \int_{\R} u_{\star} \cdot \varphi_t dx dt - \int_{\R} u_0 \varphi(x) dx = \int_0^\infty \int_{\R} \Laq (u_{\star}  \times \varphi ) \cdot \Laq u_{\star} dx dt.
    \end{split}
\end{equation}

We consider a monotone sequence of $T_n:=[0,n]$ s.t. $T_n \rightarrow \infty$ as $n\rightarrow \infty$, and $U_n\subseteq \R$ be a sequence of open interval s.t. $U_n \rightarrow \R$ as $n\rightarrow \infty$.  
On each local domain $T_n\times U_n$, from Theorem \ref{thm:globalh12}, we know there exists a solution $u_{n,\varepsilon} \in L_{t,x}^\infty \cap C_t^0([0,T_n],\dot{H}^{\frac{1}{2}}(U_n))$.  We will show that there exists a corresponding weak solution $u_{n,\star}$ for \eqref{eq:reg} s.t. $u_{n,\varepsilon} \rightharpoonup u_{n,\star}$ in $L_{t}^2(T_n,\dot{H}^{\frac{1}{2}}(U_n))$.
 Then we use a Contour argument to show that there exists a $u_{\star}$ that is a weak solution for \eqref{eq3:halfwave} on $[0,\infty) \times \R$ as defined in \eqref{weaksol}.

\subsection{Local Weak Solution}
For each local domain $T_n\times U_n$, we want to show that as $\varepsilon \rightarrow 0$ (up to a subsequence), 
we have 

\begin{equation}
    \label{eq:weak1}
    \int_{T_n} \int_{U_n} u_{\varepsilon} \cdot \varphi_t dx dt  \rightarrow  \int_{T_n} \int_{U_n} u_{\star} \cdot \varphi_t dx dt,
\end{equation}

\begin{equation}
    \label{eq:weak2}
    \varepsilon \int_{T_n} \int_{U_n} u_{\varepsilon}  \cdot \triangle \varphi dx dt  \rightarrow 0,
\end{equation}

and
\begin{equation}
    \label{eq:weak3}
    \int_{T_n} \int_{U_n} \Laq (u_{\varepsilon}  \times \Laq \varphi  ) \cdot u_{\varepsilon} dx dt  \rightarrow  \int_{T_n} \int_{U_n} \Laq (u_{\star}  \times \varphi ) \cdot \Laq u_{\star}dx dt .
\end{equation}

We first introduce the following compactness lemma for the $\dot{H}^{\frac{1}{2}}(\R)$ space.


\begin{Lemma}(Theorem 7.1 in \cite{di2012hitchhiker})
    \label{thm:frac}
    Let $s\in (0,1)$, $p\in [1,\infty)$, and $q\in [1,p)$, $\Omega \subset \R^n$ be a bounded extension domain of $W^{s,p}$ and $\mathcal{F}$ be a bounded subset of $L^p(\Omega)$. Suppose that 
    \begin{equation*}
        \sup_{f\in \mathcal{F}} \|f\|_{\dot{W}^{s,p}} <\infty
    \end{equation*}
    Then $\mathcal{F}$ is pre-compact in $L^q(\Omega)$.
\end{Lemma}

From the above lemma, we know that $u_{\varepsilon}(t,\cdot)$ is pre-compact in $L^2(U_n)$ for any $t\in T_n$. Hence we know that there exist a $u_{\star}(t,\cdot) \in \dot{H}^{\frac{1}{2}}(U_n)$ s.t. $u_{\varepsilon}(t,\cdot) \rightarrow u_{\star}(t,\cdot)$ in $L^2(U_n)$ for any $t\in T_n$ as $\varepsilon \rightarrow 0$.

In order to show \eqref{eq:weak1}, we further split $u_\varepsilon$ in the frequency domain. Given a $N>0$, we have 
\[
u_{\varepsilon}= u_{<N,\varepsilon} + u_{\geq N,\varepsilon}.
\]

We consider the low-frequency part first. We will show that $u_{<N,\varepsilon}$ has better space-time regularity which implies the strong convergence of $u_{<N,\varepsilon}$ in $L^2_{t,x}(T_n \times U_n)$.

\begin{Lemma}
    \label{lem:timereg}
For $T>0$, $U \subseteq \R$ be a bounded domain and $N>0$. For the collections of solutions $u_{\varepsilon} \in C_0^t([0,T], \dot{H}^\frac{1}{2}(U))$ of equation \eqref{eq:reg}, we know there exists a subsequence s.t. $u_{<N,\varepsilon_k} \rightarrow u_{<N,\star}$ in $L_{t,x}^2([0,T]\times U)$. 
\end{Lemma}

\begin{proof}

For $u_\varepsilon$ satisfies \eqref{eq:reg}, we have

\begin{equation}
    P_{<N} \pt u_{\varepsilon} - \varepsilon P_{<N}\triangle u= P_{<N} (u_{\varepsilon} \times \Lah u_{\varepsilon} )
\end{equation}


Hence, we have
\begin{equation}
    \begin{split}
    &\| P_{<N} \pt u_{\varepsilon}\|_{L_t^2 L_x^2} \\
    &\leq \varepsilon \| P_{<N} \triangle u_{\varepsilon}\|_{L_t^2 L_x^2} + \| P_{<N} (u_{\varepsilon} \times \Lah u_{\varepsilon} )\|_{L_t^2 L_x^2}\\
    \end{split}
\end{equation}

We know that

\begin{equation}
    \begin{split}
        &\| P_{<N} \triangle u_{\varepsilon}\|_{L_t^2 L_x^2(T\times U)}\\
    &\leq N^2 \| u_{\leq N,\varepsilon}\|_{L_t^2 L_x^2(T\times U)}\\
    &\leq C.
    \end{split}
\end{equation}

Next, we consider the nonlinear part. We use Bernstein's inequality to obtain

\begin{equation}
    \begin{split}
&\big\|P_{<N}(u_{\epsilon}\times (-\triangle)^{1/2}u_{\epsilon})\big\|_{L_x^2}\\
&\leq C2^{\frac{N}{2}}\cdot \big\|P_{<N}(u_{\epsilon}\times (-\triangle)^{1/2}u_{\epsilon})\big\|_{L_x^1}. 
    \end{split}
\end{equation}

We further split this into two terms:
\begin{equation}
    \begin{split}
 &P_{<N}(u_{\epsilon}\times (-\triangle)^{1/2}u_{\epsilon})\\
=& P_{<N}(u_{\epsilon}\times P_{<N+10}(-\triangle)^{1/2}u_{\epsilon})\\
&+P_{<N}(u_{\epsilon}\times P_{\geq N+10}(-\triangle)^{1/2}u_{\epsilon})
    \end{split}
\end{equation}

We estimate the first term by 
\begin{equation}
    \begin{split}
        &\| P_{<N}(u_{\epsilon}\times P_{<N+10}(-\triangle)^{1/2}u_{\epsilon})\|_{L_t^2 L_x^1} \\
        &\lesssim N 2^N  \|u_{\epsilon}\|^2_{L_t^2L_x^2}
    \end{split}
\end{equation}

For the second term, we write further

\begin{equation}
    \begin{split}
        &P_{<N}(u_{\epsilon}\times P_{\geq N+10}(-\triangle)^{1/2}u_{\epsilon})\\
=& \sum_{k_1 = k_2+O(1)\geq N}P_{<N}(P_{k_1}u_{\epsilon}\times P_{k_2}(-\triangle)^{1/2}u_{\epsilon}).
    \end{split}
\end{equation}

For a fixed time $t$, we have

\begin{equation}
    \begin{split}
        &\sum_{k_1 = k_2+O(1)\geq N}\big\|(P_{k_1}u_{\epsilon}\times P_{k_2}(-\triangle)^{1/2}u_{\epsilon})\big\|_{L_x^1}\\
         &\leq \sum_{k_1 = k_2+O(1)\geq N} \big\|\Laq u_{k_1,\epsilon}\big\|_{L_x^2}\cdot \big\|\Laq u_{k_2,\epsilon}\big\|_{L_x^2} \\
        &\leq \sum_{k\geq N} \big\|(-\triangle)^{1/4} u_{k,\epsilon}\big\|_{L_x^2}^2 \\
        &\leq C \|u_0\|_{\dot{H}^\frac{1}{2}}.
    \end{split}
\end{equation}

Then we can conclude that

\begin{equation}
    \begin{split}
        &\| P_{<N} (u_{\varepsilon} \times \Lah u_{\varepsilon} )\|_{L_t^2 L_x^2(T\times U)}\leq C,
    \end{split}
\end{equation}
where $C$ is a constant depending on $N$ and $u_0$ but independent of $\varepsilon$. Therefore, we know that $\|P_{<N}\pt u_{\varepsilon}\| \in L_t^2L^2_x([0,T]\times U)$. 

Since our solution $u_\varepsilon$ is also in $\dot{H}^1$, we know that $\nabla_{t,x} u_{<N,\varepsilon}$ are uniformly bounded in  $L^2_{t,x}(T_n \times U_n)$. Moreover, $\dot{H}^1$ compactly embeds into $L^2$ on $U_n$, we conclude that there exists a subsequence $\{u_{<N,\varepsilon_k}\}$ such that 
\[
    u_{<N,\varepsilon_k} \rightarrow  u_{<N,\star}  \quad \text{in}\  L^2_{t,x}(T \times U) \ \text{as} \ \varepsilon_k \rightarrow 0.
\]

\end{proof}

For the large frequency part, we have 
\begin{equation}
    \label{eq:highcov}
    \begin{split}
        &\|u_{\geq N,\varepsilon}\|_{L_{t,x}^2(T_n \times U_n)} \leq \frac{|T_n|}{N^{\frac{1}{2}}} \|u_{\varepsilon}\|_{L_t^\infty (T_n) \dot{H}^{\frac{1}{2}}(U_n)} \leq \frac{|T_n|}{N^{\frac{1}{2}}} \|u_0\|_{\dot{H}^{\frac{1}{2}}(U_n)},
    \end{split}
\end{equation}
which is a Cauchy sequence in terms of $N$. We can also conclude a similar result for $u_{>N,\star}$.

%

Therefore, we can further subtract a subsequence of $\varepsilon$ depending on $N$, for instance, $\varepsilon_N=\frac{1}{N} \rightarrow 0$ as $N\rightarrow \infty$ to obtain that 
\begin{equation}
    \begin{split}
        \|u_{\varepsilon_N} - u_{\star}\|_{L_{t,x}^2} &\leq \|u_{\varepsilon_N} - u_{<N,\varepsilon_N}\|_{L_{t,x}^2} + \|u_{<N,\varepsilon_N}- u_{<N,\star}\|_{L_{t,x}^2} +\|u_{<N,\star} -u_\star\|_{L_{t,x}^2} \\
        &\leq \|u_{\geq N,\varepsilon_N}\|_{L_{t,x}^2}+\|u_{<N,\varepsilon_N}- u_{<N,\star}\|_{L_{t,x}^2} + \|u_{\geq N,\star}\|_{L_{t,x}^2} \\
       & \rightarrow 0, \quad \text{as } N\rightarrow \infty.
    \end{split}
\end{equation}

Thus we have

\begin{equation}
    \label{eq:gconv}
    u_{\varepsilon} \rightarrow u_{\star} \quad \text{in}\ L_{t,x}^2(T_n \times U_n).
\end{equation}

So \eqref{eq:weak1} holds locally on $T_n \times U_n$.


From the maximum principle Theorem \ref{thm:maximum_v}, we know that \eqref{eq:weak2} holds. 

For the last term \eqref{eq:weak3}, we first use the cancellation property that
\begin{equation*}
    ( \Laq  u_{\varepsilon}  \times \varphi )\cdot \Laq u_{\varepsilon}  =  0,
\end{equation*}

to reformulate \eqref{eq:weak3} as
\begin{equation}
    \label{eq:weak31}
    \begin{split}
        &\int_{T_n} \int_{\R} (u_{\varepsilon}  \times \Lah u_{\varepsilon} ) \cdot \varphi dx dt\\
        &=\int_{T_n} \int_{\R} (\varphi \times u_{\varepsilon} ) \cdot \Lah u_{\varepsilon}  dx dt\\
        &=\int_{T_n} \int_{\R} \Laq ( u_{\varepsilon}  \times \varphi) \cdot \Laq u_{\varepsilon}  \, dx dt\\
        & = \int_{T_n} \int_{\R} (\Laq ( u_{\varepsilon}  \times \varphi)-(\Laq u_{\varepsilon}\times \varphi ) ) \cdot \Laq u_{\varepsilon}  \, dx dt 
    \end{split}
\end{equation}


We know that for a fixed $\varphi \in C_c^\infty$, the frequency $\eta$ of $\hat{\varphi}$ decays fast for large $|\eta|>N_\varepsilon$.

We choose a $N>N_{\varepsilon}$, and then we split the \eqref{eq:weak31} into two parts:

\begin{align}
    &\int_{T_n} \int_{\R} (\Laq ( u_{\varepsilon}  \times \varphi)-(\Laq u_{\varepsilon}\times \varphi ) ) \cdot \Laq u_{\varepsilon}  \, dx dt \nonumber \\
    &=\int_{T_n} \int_{\R} (\Laq ( u_{<N,\varepsilon}  \times \varphi)-(\Laq u_{<N,\varepsilon}\times \varphi ) ) \cdot \Laq u_{\varepsilon}  \, dx dt \label{eq:weak3small}\\ 
    &+\int_{T_n} \int_{\R} (\Laq ( u_{\geq N,\varepsilon}  \times \varphi)-(\Laq u_{\geq N,\varepsilon}\times \varphi ) ) \cdot \Laq u_{\varepsilon}  \, dx dt \label{eq:weak3large}
\end{align}


For a given $t\in T_n$, we first consider the small frequency part \eqref{eq:weak3small}. We further separate it into two cases: $|\eta|\leq |\xi| <N$ and $|\xi|\leq |\eta|<N$. We have

\begin{equation}
    \label{eq:smallN1}
    \begin{split}
        &\|  \Laq (u_{\varepsilon, <N} \times \varphi) - (\Laq u_{\varepsilon,<N} \times \varphi)  \|_{L_x^2} \\
        &\lesssim \sum_{|\eta|<|\xi|<N}\| (|\xi+\eta|^{\frac{1}{2}}-|\xi|^{\frac{1}{2}}) \hat{u}_{\varepsilon}(\xi) \star \hat{\varphi}(\eta) \|_{L^2}\\
        &+\sum_{|\xi|<|\eta|<N}\| (|\xi+\eta|^{\frac{1}{2}}-|\xi|^{\frac{1}{2}}) \hat{u}_{\varepsilon}(\xi) \star \hat{\varphi}(\eta) \|_{L^2}\\
    \end{split}
\end{equation}

We estimate the first part as
\begin{equation}
    \label{eq:smallN2}
    \begin{split}
        & \sum_{|\eta|<|\xi|<N}\| (|\xi+\eta|^{\frac{1}{2}}-|\xi|^{\frac{1}{2}}) \hat{u}_{\varepsilon}(\xi) \star \hat{\varphi}(\eta) \|_{L_x^2}\\
        &\lesssim  \sum_{|\eta|<|\xi|<N} \| \frac{|\eta|}{|\xi|}  |\xi|^{\frac{1}{2}} \hat{u}_{\varepsilon}(\xi) \star \hat{\varphi}(\eta) \|_{L_x^2}\\
        &\lesssim  \sum_{|\eta|<|\xi|<N} \|  |\xi|^{\frac{1}{2}} \hat{u}_{<N,\varepsilon}(\xi) \star \hat{\varphi}(\eta) \|_{L^2}\\
        &\lesssim \|u_{<N,\varepsilon}\|_{L^2_x} \| \varphi \|_{L_x^1} 
    \end{split}
\end{equation}

For the second part, we have 
\begin{equation} 
    \label{eq:smallN3}
    \begin{split}
       &\sum_{|\xi|<|\eta|<N}\| (|\xi+\eta|^{\frac{1}{2}}-|\xi|^{\frac{1}{2}}) \hat{u}_{\varepsilon}(\xi) \star \hat{\varphi}(\eta) \|_{L^2}\\
        &\lesssim  \sum_{|\xi|<|\eta|<N} \| |\eta|^{\frac{1}{2}} \hat{u}_{\varepsilon}(\xi) \star \hat{\varphi}(\eta) \|_{L^2}\\
        &\lesssim \|u_{\varepsilon, <N}\|_{L_x^2} \| \varphi \|_{L_x^1}
    \end{split}
\end{equation}

Combine Lemma \ref{lem:timereg} with estimates \eqref{eq:smallN2} and \eqref{eq:smallN3}, we have
\begin{equation}
    \begin{split}
        &\|\Laq (u_{<N,\varepsilon} \times \varphi) - \Laq u_{<N,\varepsilon} \times \varphi -\big(\Laq (u_{<N,\star} \times \varphi) - \Laq u_{<N,\star} \times \varphi\big)\|_{L_{t,x}^2(T_n\times U_n)}\\
        &=\| \Laq ((u_{<N,\varepsilon}-u_{<N,\star}) \times \varphi) - \Laq (u_{<N,\varepsilon}-u_{<N,\star}) \times \varphi\|_{L_{t,x}^2(T_n\times U_n)}\\
    &\lesssim \|u_{<N,\varepsilon} - u_{<N,\star}\|_{L_{t,x}^2(T_n\times U_n)} \ \| \varphi \|_{L^1} \rightarrow 0 \quad \text{as} \ \varepsilon \rightarrow 0
    \end{split}
\end{equation}

By the a priori bound \eqref{eq:h12uni}, we know that up to a subsequence, 

\[
    \Laq u_{\varepsilon} \rightharpoonup \Laq u_{\star} \quad \text{in} \ L^2
\]

Hence, on each compact domain $T_n\times U_n$, we have

\begin{equation}
\int_{T_n} \int_{U_n} (u_{<N,\varepsilon}  \times \Lah u_{<N,\varepsilon} ) \cdot \varphi dx dt  \rightarrow \int_{T_n} \int_{U_n} (u_{<N,\star}  \times \Lah u_{<N,\star} ) \cdot \varphi dx dt
\end{equation}

Next, we consider the large frequency part \eqref{eq:weak3large}. We have a better cancellation here. 

\begin{equation}
    \begin{split}
        & \|  \Laq (u_{>N,\varepsilon} \times \varphi) - (\Laq u_{>N,\varepsilon} \times \varphi)  \|_{L_x^2} \\
        &\lesssim \sum_{|\eta|<N<|\xi|}\| (|\xi+\eta|^{\frac{1}{2}}-|\xi|^{\frac{1}{2}}) \hat{u}(\xi) \star \hat{\varphi}(\eta) \|_{L^2}\\
        &\lesssim  \sum_{|\eta|<N<|\xi|} \| \frac{|\eta|}{ |\xi|^{\frac{1}{2}}}  \hat{u}(\xi) \star \hat{\varphi}(\eta) \|_{L^2}\\
        &\lesssim  \frac{1}{N} \ \sum_{|\eta|<N<|\xi|} \|  |\xi|^{\frac{1}{2}} \hat{u}(\xi) \star |\eta| \hat{\varphi}(\eta) \|_{L^2}\\
        &\lesssim \frac{1}{N} \| u_{>N,\varepsilon} \|_{\dot{H}^{\frac{1}{2}}} \lesssim \frac{1}{N} \| u_0 \|_{\dot{H}^{\frac{1}{2}}} 
    \end{split}
\end{equation}

Integrating over $T_n$, we have
\begin{equation}
    \begin{split}
        &\|  \Laq (u_{>N,\varepsilon} \times \varphi) - (\Laq u_{>N,\varepsilon} \times \varphi)  \|_{L_{t,x}^2(T_n\times U_n)} \\
        &\lesssim \frac{|T_n|}{N} \| u_0 \|_{\dot{H}^{\frac{1}{2}}},
    \end{split}
\end{equation}
is a Cauchy sequence in terms of $N$.
 
Hence, we can further choose a subsequence of $\varepsilon_N \rightarrow 0$ as $N\rightarrow \infty$, such that
\[
    \int_{T_n} \int_{U_n} (u_{\varepsilon}  \times \Lah u_{\varepsilon} ) \cdot \varphi dx dt  \rightarrow \int_{T_n} \int_{U_n} (u_{\star}  \times \Lah u_{\star} ) \cdot \varphi dx dt
\]


Thus, we conclude that, up to a subsequence, \eqref{eq:weak3} holds. Hence we know that on each local domain $T_n\times U_n$, we have a weak solution $u_{\star}$ for the half-wave map equation \eqref{eq3:halfwave}.

\subsection{Global Weak Solution}

For each local doman $\Omega_n=T_n \times U_n$, s.t.  $\Omega_n \rightarrow \R^{1+1}$ as $n \rightarrow \infty$, we know that for each $\Omega_n$, there exist a subsequence $\varepsilon_{n,k} \rightarrow 0$, s.t. $u_{\varepsilon_{n,k}}$ weakly converging to a $u_{n,\star}$ in $L_t^2(T_n, \dot{H}^{\frac{1}{2}}(U_n))$.

We then use Cantor's diagonal argument to conclude that there exists a diagonal sequence $\varepsilon_{n,n}$, s.t.
$u_{\varepsilon_{n,n}}$ weakly converging to a $u_{\star}$ in $L_{t,loc}^2([0,\infty), \dot{H}^{\frac{1}{2}} (\R))$. For the first $\Omega_1$, we pick a subsequence $\varepsilon_{1,n}$, s.t. $u_{\varepsilon_{1,n}}$ weakly converging to $u_{1,\star}$ in $L_t^2 (T_1,\dot{H}^{\frac{1}{2}} (U_1))$. Next, we consider a subsequence $\varepsilon_{2,n}$ of $\varepsilon_{1,n}$, s.t.  $u_{\varepsilon_{2,n}}$ weakly converging to $u_{2,\star}$ in $L_t^2 (T_2,\dot{H}^{\frac{1}{2}} (U_2))$. We keep doing this process for all $n\in \mathbb{N}$. Finally, we pick a diagonal sequence $\varepsilon_{n,n}$ so that $u_{\varepsilon_{n,n}}$ weakly converging to $u_{\star}$ in $L_{t,loc}^2([0,\infty), \dot{H}^{\frac{1}{2}} (\R))$. Therefore, we obtained our global weak solution $u_{\star}$.


Lastly, we verify that $u_{\star}$ maps into $S^2$. 
We consider the equation \eqref{eq:v}. By the heat kernel, we have 
\begin{equation}
    v_\varepsilon = -\varepsilon K_\varepsilon \star |\nabla u_\varepsilon|^2 +1.
\end{equation}

Since $u_\varepsilon(t) \in \dot{H}_x^1$, we know that $\nabla u_\varepsilon(t) \in L_x^2$ for all $t \in [0,\infty)$. So as $\varepsilon \rightarrow 0$, we know that $v_\varepsilon(t) \rightarrow 1$ for all $t \in [0,\infty)$.

From \eqref{eq:gconv}, we know that $u_\varepsilon \rightarrow u_\star$ strongly in $L_{t,x}^2(T_n\times U_n)$. Hence, we have

\begin{equation}
    u_{\varepsilon} \rightarrow u_{\star} \ \text{almost everywhere on } T_n \times U_n.
\end{equation}

With the same diagonal argument for $u_\varepsilon$ above, we have the diagonal subsequence $\varepsilon_{n,n} \rightarrow 0$, s.t.
\[
  u_{\varepsilon_{n,n}} \cdot u_{\varepsilon_{n,n}} \rightarrow u_\star \cdot u_\star = 1 \ \text{almost everywhere on } \R^{1+1},
\]
which implies that $u_\star$ maps into $S^2$.


\newpage

\newpage
\bibliographystyle{plain} 
\bibliography{paper}
\end{document}